\documentclass[12pt]{amsart}
\usepackage[utf8]{inputenc}
\usepackage{amssymb}
\usepackage{amsmath}
\usepackage{amsfonts}
\usepackage{amsthm}
\usepackage{tikz}
\usepackage{algorithmic}
\usepackage{asymptote}
\usepackage{url}
\newtheorem{theorem}{Theorem}[section]
\newtheorem{corollary}{Corollary}[theorem]
\newtheorem{lemma}[theorem]{Lemma}
\newtheorem{proposition}[theorem]{Proposition}
\newtheorem{conjecture}[theorem]{Conjecture}
\newtheorem{definition}[theorem]{Definition}

\usepackage[OT2,T1]{fontenc}
\DeclareSymbolFont{cyrletters}{OT2}{wncyr}{m}{n}
\DeclareMathSymbol{\Sha}{\mathalpha}{cyrletters}{"58}

\theoremstyle{definition}

\newtheorem{example}[theorem]{Example}

\newcommand{\N}{\mathbb{N}}
\newcommand{\Nd}{\mathbb{N}^d}
\newcommand{\Hs}{\mathcal{H}}

\newtheorem{innercustomthm}{Theorem}
\newenvironment{customthm}[1]
  {\renewcommand\theinnercustomthm{#1}\innercustomthm}
  {\endinnercustomthm}

\newtheorem{innercustomprop}{Proposition}
\newenvironment{customprop}[1]
  {\renewcommand\theinnercustomprop{#1}\innercustomprop}
  {\endinnercustomprop}

\begin{document}
\title{Frobenius allowable gaps of Generalized Numerical Semigroups}
\author{Deepesh Singhal}
\address{University of California, Irvine}
\email{singhald@uci.edu}
\author{Yuxin Lin}
\address{University of Notre Dame, USA}
\email{ylin9@nd.edu}

\begin{abstract}
    A generalised numerical semigroup (GNS) is a submonoid $S$ of $\Nd$ for which the complement $\Nd\setminus S$ is finite. The points in the complement $\Nd\setminus S$ are called gaps. A gap $F$ is considered Frobenius allowable if there is some relaxed monomial ordering on $\Nd$ with respect to which $F$ is the largest gap. We characterise the Frobenius allowable gaps of a GNS.
    A GNS that has only one Frobenius allowable gap is called a Frobenius GNS. We estimate the number of Frobenius GNS with a given Frobenius gap $F=(F^{(1)},\dots,F^{(d)})\in\Nd$ and show that it is close to $\sqrt{3}^{(F^{(1)}+1)\cdots (F^{(d)}+1)}$ for large $d$.
    We define notions of quasi-irreducibility and quasi-symmetry for GNS. While in the case of $d=1$ these notions coincide with irreducibility and symmetry of GNS, they are distinct in higher dimensions.
    
\end{abstract}

\maketitle

\section{Introduction}
A \emph{numerical semigroup} $S$ is a subset of the natural numbers that contains $0$, is closed under addition and has a finite complement $\N\setminus S$. The numbers in $\N\setminus S$ are called \emph{gaps} and the largest gap is called the \emph{Frobenius number} $F(S)$.
There is a big literature on numerical semigroups, see \cite{Gen ref 1}, \cite{Gen ref 2} for a general reference.

A generalised numerical semigroup (GNS) $S$ is a subset of $\Nd$ that contains $0$, is closed under addition and has a finite complement $\Nd\setminus S$. The points in the complement are called \emph{gaps} and the collection of all gaps is denoted by $\Hs(S)=\Nd\setminus S$. The number of gaps is called the genus, $g(S)=|\Hs(S)|$. Failla, Peterson and Utano gave this definition of GNS in \cite{Def F}. They also studied the question of counting GNSs by genus and generalised the notion of semigroup tree.

We have a natural partial ordering on $\Nd$. Given $x,y\in \Nd$, let $x^{(i)},y^{(i)}$ be the $i^{th}$ component of $x, y$ respectively. Define $x\leq y$ if $x^{(i)}\leq y^{(i)}$ for each $1\leq i\leq d$. However this is not enough to define the Frobenius gap of $S$ as $\Hs(S)$ could have more than one maximal elements under the natural partial ordering.
Failla et al. \cite{Def F} extend the notion of Frobenius gap to generalised numerical semigroups with the help of \emph{relaxed monomial orderings} on $\Nd$.
\begin{definition}
A total order $\prec$ on the elements of $\mathbb{N}^d$ is called a relaxed monomial order if it satisfies:\\
i) If $v,w\in\mathbb{N}^d$ and if $v\prec w$ then $v\prec w+u$ for any $u\in\mathbb{N}^d$.\\
ii) If $v\in\mathbb{N}^d$ and $v\neq 0$ then $0\prec v$.
\end{definition}
Given a relaxed monomial order $\prec$ on $\Nd$ and a GNS $S\subseteq\Nd$, its Frobenius gap is defined as $$F_{\prec}(S)=\max_{\prec}(\Hs(S)).$$
Of course, different relaxed monomial orders can lead to different gaps becoming the Frobenius gap of $S$.  Cisto, Failla, Peterson and Utano in \cite{Irr} define a gap of $S$ to be \emph{Frobenius allowable} if it is the Frobenius gap with respect to some relaxed monomial ordering. It is clear that all Frobenius allowable gaps must be maximal elements of $\Hs(S)$ under the natural partial ordering. In \cite{Irr} the authors ask whether all maximal elements of $\Hs(S)$ under the natural partial ordering are Frobenius allowable. They prove this (Proposition 4.5 and 4.13, \cite{Irr}) in the special case when $\Hs(S)$ has exactly one or two maximal elements under the natural partial ordering. We answer their question in the general case.

\begin{customthm}{\ref{Frobenius allowable}}
Given a GNS $S\subseteq\Nd$ the Frobenius allowable gaps of $S$ are precisely the maximal elements of $\Hs(S)$ under the natural partial ordering.
\end{customthm}

Cisto et al. in \cite{Irr} define a Frobenius GNS to be a GNS $S$ for which $\Hs(S)$ has exactly one maximal gap under the natural partial ordering. This is equivalent to the property that the Frobenius gap of $S$ being independent of the choice of relaxed monomial ordering. The authors of \cite{Irr, Almost Irr} study certain families of GNS which they show are Frobenius GNS.

If one fixes a point $F\in\Nd\setminus\{0\}$ with $d\geq 2$ then it is seen that there are infinitely many GNS for which $F$ is Frobenius allowable. However the number of Frobenius GNS with a given Frobenius gap $F$ is clearly finite. We denote this by
$$N(F)=\#\{S\subseteq\Nd\mid S\text{ is a Frobenius GNS}, F(S)=F\}.$$
In the case of numerical semigroups i.e. $d=1$, Backelin \cite{Backelin} computes $N(F)$ and proves that
$$2^{\lfloor\frac{F-1}{2}\rfloor}\leq N(F)\leq 4\times 2^{\lfloor\frac{F-1}{2}\rfloor}.$$
We build on the work of Backelin and make the first systematic study of counting Frobenius GNS with a given Frobenius number.

Given a point $F$ in $\Nd$, let
$$\|F\|=\prod_{i=1}^{d}(F^{(i)}+1).$$
So $\|F\|$ is the number of points in the box $\{x\in\Nd\mid 0\leq x\leq F\}$. We trivially know that $N(F)\leq 2^{\|F\|}$. For large $d$ we prove that $N(F)$ is quite close to $\sqrt{3}^{\|F\|}$.

\begin{customthm}{\ref{count large d}}
Given $\epsilon>0$, there exist $M>0$ such that for every $d>M$ and $F\in\Nd$, we have
$$\left(\sqrt{3}-\epsilon\right)^{\|F-1\|}\leq N(F)\leq \sqrt{3}^{\|F\|}.$$
\end{customthm}

Given a GNS $S\subseteq\Nd$, a gap $P\in\Hs(S)$ is called a pseudo-Frobenius gap of $S$ if for nonzero every element $x$ of $S$, $x+P$ is also an element of $S$. The collection of all pseudo-Frobenius gaps of $S$ is denoted by $PF(S)$. And the number of pseudo-Frobenius gaps of $S$ is called its type, $t(S)$.

The family of irreducible numerical semigroups have received considerable attention in the literature. A numerical semigroup is called irreducible if it cannot be expressed as the intersection of two numerical semigroups properly containing it. Several characterisations of irreducible numerical semigroups are known. The authors of \cite{Type} prove that given a numerical semigroup $S$ with $F(S)=F$, the following are equivalent
\begin{itemize}
    \item $S$ is irreducible.
    \item $S$ is maximal (with respect to set theoretic inclusion) among all numerical semigroups that do not contain $F$.
    \item For every gap $x\in\Hs(S)$, either $2x=F(S)$ or $F(S)-x\in S$.
    \item Either $PF(S)=\{F\}$, or $PF(S)=\{\frac{F}{2},F\}$.
\end{itemize}
A numerical semigroup with $PF(S)=F$ is called symmetric and one with $PF(S)=\{\frac{F}{2},F\}$ is called pseudo-symmetric. A numerical semigroup has type $1$ if and only if it is symmetric.

The authors of \cite{Irr} define a GNS to be irreducible if it cannot be written as an intersection of two numerical semigroups properly containing it. They prove that irreducible GNS are always Frobenius GNS. They prove that the following are equivalent for a GNS $S$:
\begin{itemize}
    \item $S$ is irreducible.
    \item There is a gap $F\in\Hs(S)$ such that for every gap $x\in\Hs(S)$, either $2x=F(S)$ or $F(S)-x\in S$.
    \item There is a gap $F\in\Hs(S)$ such that either $PF(S)=\{F\}$ or $PF(S)=\{\frac{F}{2},F\}$.
\end{itemize}
They call a GNS symmetric if $PF(S)=\{F\}$ and pseudo-symmetric if $PF(S)=\{\frac{F}{2},F\}$. This shows that a GNS is symmetric if and only if it has $t(S)=1$. The authors of \cite{Irr} also show that every GNS can be written as a finite intersection of irreducible GNS.

By Theorem \ref{Frobenius allowable} we know that Frobenius allowable gaps of $S$ are maximal in $\Hs(S)$ under the natural partial ordering. Therefore they are all pseudo-Frobenius gaps. We denote the collection of all Frobenius allowable gaps of $S$ by $FA(S)$ and number of Frobenius allowable gaps by $\tau(S)$. Therefore
$$\tau(S)=|FA(S)| =\#\{x\in\Hs(S)\mid x\text{ is Frobenius allowable}\}.$$
Since all Frobenius allowable gaps are Pseudo-Frobenius, we see that $\tau(S)\leq t(S)$. We call a GNS $S$ quasi-symmetric if $\tau(S)=t(S)$.
In the case of Frobenius GNS, $\tau(S)=1$. Hence a Frobenius GNS $S$ will be quasi-symmetric when $t(S)=1$. Thus the property of being symmetric is equivalent to being both quasi-symmetric and a Frobenius GNS. The notions of symmetry and quasi-symmetry of course coincide in the case of numerical semigroups.
In Theorem \ref{tau=t} we show that quasi-symmetric GNS are characterised by a property similar to that of symmetric numerical semigroups.
\begin{customthm}{\ref{tau=t}}
Given a GNS $S\subseteq\Nd$, $\tau(S)\leq t(S)$. Moreover equality holds if and only if $S$ satisfies the property that for every $x\in\Hs(S)$ there is some Frobenius allowable gap $F$ for which $F-x\in S$.
\end{customthm}

We call a GNS $S$ quasi-irreducible if for every gap $x\in\Hs(S)$ either $2x$ is Frobenius allowable or there is some $F\in FA(S)$ for which $F-x\in S$. From Theorem \ref{tau=t}, we see that all quasi-symmetric GNS are quasi-irreducible. In Theorem \ref{quasi irreducible} and Proposition \ref{quasi-irreducible PF} we prove the following.
\begin{theorem}\label{char Irr}
Given a GNS $S$, the following are equivalent:
\begin{itemize}
    \item $S$ is quasi-irreducible, i.e. for every gap $x\in\Hs(S)$ either $2x\in FA(S)$ or there is some $F\in FA(S)$ for which $F-x\in S$.
    \item $S$ is maximal in the collection of all GNS $S'$ for which $FA(S)\subseteq \Hs(S')$.
    \item For all $P\in PF(S)$, either $P\in FA(S)$ or $2P\in FA(S)$.
\end{itemize}
\end{theorem}

We also study the bounds on the type of a Frobenius GNS.
For numerical semigroups it is known that
$$\frac{g(S)}{F(S)+1-g(S)}\leq t(S)\leq 2g(S)-F(S).$$
The first inequality was proved in \cite{Type}, and the second in \cite{H Nari}. Numerical semigroups that satisfy $t(S)=2g(s)-F(S)$ are called almost-symmetric.
In \cite{Almost Irr}, the authors extended the second inequality. They prove that if $S$ is a Frobenius GNS then
$$t(S)\leq 2g(S)+1-\|F(S)\|.$$ 
If equality holds then they call Frobenius GNS almost-symmetric. The authors of \cite{Almost Irr} come up with a number of equivalent conditions for a Frobenius GNS to be almost-symmetric. We give another property that is equivalent to almost-symmetry.
The special case of Proposition \ref{Condition almost sym} for numerical semigroups was proved in \cite{My paper type}.
\begin{customprop}{\ref{Condition almost sym}}
A Frobenius GNS $S$ with Frobenius gap $F$ is almost symmetric if and only if
$$T(S)=\{x\in\Nd\mid F-x\in \left(\left(\mathbb{Z}^d\setminus S\right)\cup\{0\}\right)\}$$
is a GNS.
\end{customprop}
We also extend the first inequality as follows:
\begin{customthm}{\ref{lower bound for type}}
Given a Frobenius GNS $S\subseteq\Nd$ we have
$$\frac{g(S)}{\|F(S)\|-g(S)}\leq t(S).$$
\end{customthm}

\section{Frobenius Allowable Gaps}

In this section we will prove Theorem \ref{Frobenius allowable}. We fix a GNS $S\subseteq\Nd$ and an element $h$ of $\Hs(S)$, which is maximal under the natural partial ordering. We construct an explicit relaxed monomial order with respect to which $h$ becomes the Frobenius gap of $S$.

\begin{theorem}\label{Frobenius allowable}
Given a GNS $S\subseteq\Nd$ the Frobenius allowable gaps of $S$ are precisely the maximal elements of $\Hs(S)$ under the natural partial ordering.
\end{theorem}
\begin{proof}
Let $h$ be a maximal element of $\Hs(S)$ under the natural partial ordering. We reorder the coordinates so that $h^{(1)},h^{(2)},\dots,h^{(k)}$ are all non zero and $h^{(k+1)},h^{(k+2)},\dots,h^{(d)}$ are zero. We then define a function $\phi$ on $\mathbb{N}^d$:
$$\phi(x)=\min_{1\leq i\leq k}\left(\frac{x^{(i)}}{h^{(i)}}\right).$$
We now define $\prec$ as follows, suppose $x,y\in \Nd$
\begin{itemize}
    \item If $\phi(x)<\phi(y)$ then $x\prec y$.
    \item If $\phi(x)=\phi(y)$ and there is a $j\in\{0,1\dots,d-1\}$ such that $x^{(i)}=y^{(i)}$ for $1\leq i\leq j$ and $x^{(j+1)}<y^{(j+1)}$ then $x\prec y$.
\end{itemize}
This is clearly a total ordering. It is also clear that $\phi(h)=1$. Moreover, given some $x\in G(S)$ other than $h$, we know that $h\not\leq x$ as $h$ is maximal in $G(S)$. This means that there is a $i$ for which $x^{(i)}<h^{(i)}$. In this case $h^{(i)}>0$ and $i\leq k$. It follows that $\frac{x^{(i)}}{h^{(i)}}<1$ and hence $\phi(x)<1=\phi(h)$. Therefore $x\prec h$. This shows that $h$ is the maximum of $G(S)$ with respect to $\prec$. The only thing that remains to be shown is that $\prec$ is a relaxed monomial order.

We know that $\phi(0)=0$. If $v\in\mathbb{N}^d$ is non-zero then there is some $j$ for which $v^{(j)}>0$. Consider the smallest such $j$. If $j\leq k$ then $\phi(0)=0<\phi(v)$ and hence $0\prec v$. On the other hand if $j>k$ then $\phi(0)=0=\phi(v)$. For $1\leq i\leq j-1$ we have $0^{(i)}=0=v^{(i)}$ and $0^{(j)}=0<v^{(j)}$. Therefore we still get $0\prec v$.

Next suppose we have $u,v,w\in\mathbb{N}^d$ such that $v\prec w$. We know that $w^{(i)}\leq (w+u)^{(i)}$ for each $i$. Moreover this implies that $\phi(w)\leq\phi(w+u)$. Combining all of this we see that $w\preccurlyeq w+u$ and hence $v\prec w+u$.

Therefore $\prec$ is indeed a relaxed monomial ordering and $F_{\prec}=h$. We see that $h$ is Frobenius allowable and this completes the proof.
\end{proof}

We also have a notion of a monomial order which is stronger than a relaxed monomial order.

\begin{definition}
A total order $\prec$, on the elements of $\mathbb{N}^d$ is called a monomial order if it satisfies:\\
i) If $v,w\in\mathbb{N}^d$ and if $v\prec w$ then $v+u\prec w+u$ for any $u\in\mathbb{N}^d$.\\
ii) If $v\in\mathbb{N}^d$ and $v\neq 0$ then $0\prec v$.
\end{definition}

It is clear that all monomial orders are also relaxed monomial orders. However as noted in \cite{Irr}, the converse is not true. In particular the relaxed monomial order we constructed in the proof of Theorem \ref{Frobenius allowable} is not a monomial order. To see this consider the case when $d=2$, $h=(1,1)$. In this case $(1,4)\prec (2,2)$ but $(1,4)+(2,0)\succ (2,2)+(2,0)$. So it is not a monomial order. Of course it is a relaxed monomial order, so we have $(1,4)\prec (2,2)+(2,0)$.

In \cite{Robbiano}, the author proved that a general monomial order on $\Nd$ can be obtained in terms of $d$ linearly independent vectors in $\mathbb{R}^d$ as follows. Given a monomial order $\prec$ there are $v_1,v_2,\dots,v_d\in\mathbb{R}^d$, that are linearly independent. Moreover for $v,w\in\Nd$ we have $v\prec w$ if and only if there is $k\in\{1,2,\dots d-1\}$ such that for every $i$ with $1\leq i\leq k-1$ we have $\langle v,v_i\rangle=\langle w,v_i\rangle$ and $\langle v,v_k\rangle<\langle w,v_k\rangle$.

Given a GNS $S\subseteq\Nd$ and a Frobenius allowable gap $F$ one could ask if there is a monomial order $\prec$ such that $F=F_{\prec}$. This is not always the case. For example let $d=2$ and consider $$S=\mathbb{N}^2\setminus\{(0,1),(0,2),(0,3),(1,0),(2,0),(3,0),(1,1)\}.$$ This is closed under addition and hence is a GNS. The gap $(1,1)$ is maximal among the gaps in the natural partial ordering. So $(1,1)$ is Frobenius allowable. However there is no monomial order $\prec$ for which $F_{\prec}=(1,1)$.

Recall that a GNS is called a Frobenius GNS if its Frobenius gap is independent of the relaxed monomial ordering. Theorem \ref{Frobenius allowable} allows us to classify which GNS are Frobenius GNS.

\begin{theorem}
Given a GNS $S$, the following are equivalent:\\
i) $S$ is a Frobenius GNS\\
ii) $\Hs(S)$ has a unique maximal element with respect to the natural partial ordering.\\
iii) $PF(S)$ has a unique maximal element with respect to the natural partial ordering.\\
\end{theorem}
\begin{proof}
By Theorem \ref{Frobenius allowable} we know that i) and ii) are equivalent. We have also seen that the maximal members of $\Hs(S)$ under the natural partial ordering are pseudo-Frobenius gaps. This means that the maximal members of $\Hs(S)$ and $PF(S)$ under the natural partial ordering are exactly the same. This shows that ii) and iii) are equivalent.
\end{proof}

\section{Quasi-irreducible GNS}

Recall that the type of a GNS is the number of pseudo-Frobenius gaps it has i.e.
$$t(S)=|PF(S)|=\#\{P\in\Hs(S)\mid P+(S\setminus\{0\})\subseteq S \}.$$
And $\tau(S)$ is the number of Frobenius allowable gaps of $S$. Since all Frobenius allowable gaps are pseudo-Frobenius, we have $\tau(S)\leq t(S)$. We start this section by characterising quasi-symmetric GNS, i.e. those GNS for which $\tau(S)=t(S)$.

\begin{theorem}\label{tau=t}
Given a GNS $S\subseteq\Nd$, $\tau(S)\leq t(S)$. Moreover equality holds if and only if $S$ satisfies the property that for every $x\in\Hs(S)$ there is some Frobenius allowable gap $F$ for which $F-x\in S$.
\end{theorem}
\begin{proof}
We already know that $\tau(S)\leq t(S)$. We now prove the next part of the theorem.
First suppose that $S$ satisfies the given property. In this case consider some $x\in\Hs(S)$ is not Frobenius allowable. We know that there must be some Frobenius allowable gap $F$ for which $F-x\in S$. We know that $F\neq x$ as $x$ is not Frobenius allowable. Therefore $F-x$ is a nonzero element of $S$ and $x+(F-x)\not\in S$. This shows that $x$ is not a pseudo-Frobenius gap of $S$. We can conclude that $\tau(S)=t(S)$.

We now prove the other direction. Suppose that $\tau(S)=t(S)$. Assume for the sake of contradiction that $S$ does not satisfy the given property. Consider the set of all gaps that fail the property
$$A=\{x\in\Hs(S)\mid \not\exists F\in FA(S): F-x\in S\}.$$
Let $x$ be a maximal member of $A$ under the natural partial ordering. We know that $x$ is not Frobenius allowable, since $x-x=0\in S$. Moreover since $\tau(S)=t(S)$, we know that $x\not\in PF(S)$. This means that there is some nonzero element $s\in S$ such that $x+s\not\in S$. By the maximality of $x$ we know that $x+s\not\in A$. Now $x+s\in\Hs(S)$, $x+s\not\in A$, so we see that for every Frobenius-allowable gap $F$ of $S$, $F-(x+s)\not\in S$. This implies that $F-x\not\in S$ as $S$ is closed under addition. However this shows that $x\not\in A$, which is a contradiction. Therefore $A$ must be empty and $S$ satisfies the given condition.
\end{proof}

We note that if a GNS has type $1$ then it must have $\tau(S)=1$ i.e. it must be a Frobenius GNS. Moreover it must also be quasi-symmetric and hence must satisfy the condition of Theorem \ref{tau=t}. These GNS are studied in \cite{Irr} and are called symmetric GNS. A GNS is symmetric if and only if $\tau(S)=1$ and it is quasi-symmetric.

Recall that a GNS $S$ is called quasi-irreducible if for every $x\in\Hs(S)$ either $2x$ is Frobenius allowable or there is some Frobenius allowable gap $F$ for which $F-x\in S$.
Clearly all quasi-symmetric GNS are quasi-irreducible.

\begin{theorem}\label{quasi irreducible}
Let $D$ be a finite subset of $\Nd\setminus\{0\}$ that is an anti-chain with respect to the natural partial ordering. Consider the collection of all GNS $S\subseteq\Nd$ for which $D\subseteq \Hs(D)$. The maximal members of this collection are precisely the quasi-irreducible GNS $S$ with $FA(S)=D$.
\end{theorem}
\begin{proof}
First suppose that we have a quasi-irreducible GNS $S$ with $FA(S)=D$. Assume for the sake of contradiction that $S$ is not maximal in the collection. This means that there is some GNS $S'\supsetneq S$ with $D\subseteq\Hs(S')$. Consider some $x\in S'\setminus S$. Since $x\in\Hs(S)$ we know that either $2x\in FA(S)=D$ or there is some $F\in D$ for which $F-x\in S$. We know that $2x\in S'$ so $2x$ cannot be in $D$. However if there is some $F\in D$ for which $F-x\in S$, then $F-x\in S'$ and hence $F=(F-x)+x\in S'$. This is also impossible. Therefore we get a contradiction and $S$ must be maximal in the collection.

We now prove the other direction, consider some GNS $S$ which is maximal in the collection. Let
$$S_1=S\cup \{a\in\Nd\mid \forall F\in D: a\not\leq F\}.$$
It is clear that $S_1$ is a GNS and $D\cap S_1=\emptyset$, so $S_1$ is in the collection. Also $S\subseteq S_1$, so the maximality of $S$ implies that $S=S_1$. Now the fact that $S=S_1$ and $D\subseteq\Hs(S)$ imply that $FA(S)\subseteq D$ (by Theorem \ref{Frobenius allowable}). Next consider some $x\in D$. We know that $x\in\Hs(S)$ so there must be some $F\in FA(S)$ for which $x\leq F$. But then $F,x\in D$. Since $D$ is an anti-chain, this implies $x=F$. Therefore $FA(S)=D$.

Next consider
$$X=\{x\in\Hs(S)\mid 2x\not\in D,\forall F\in D: F-x\not\in S\}.$$
If $X$ is empty then $S$ will be quasi-irreducible. Therefore assume for the sake of contradiction that $X$ is non-empty. Consider some $x\in X$ that is maximal with respect to the natural partial ordering. Let $$S_2=S\cup\{x\}.$$
We will show that $S_2$ is closed under addition. Consider a non-zero $s$ in $S$. By maximality of $x$, we know that $x+s\not\in X$. Therefore either $2(x+s)\in D$ or there is some $F\in D$ for which $F-(x+s)\in S$ or $x+s\in S$. We wish to show that $x+s\in S$, so we will show that the other two possibilities are impossible.
\begin{itemize}
    \item If $2(x+s)\in D$ then $x+2s=2(x+s)-x\not\in S$. Let $y=x+2s$. Then $2y=2(x+s)+2s>2(x+s)$, since $D$ is an anti-chain this means that $2y\not\in D$. Next if there is some $F\in D$ for which $F-y\in S$, then $F-x=F-y+2s\in S$ which is impossible since $x\in X$. Therefore $\forall F\in D$: $F-y\not\in S$. This means that $y\in X$ but $x<y$ and this contradicts the maximality of $x$.
    \item If there is some $F\in D$ for which $F-(x+s)\in S$. Then $F-x=F-(x+s)+s\in S$ and this contradicts the fact that $x\in X$.
\end{itemize}
Therefore we have shown that for any non-zero $s\in S$, $x+s$ is also an element of $S$. Next by the maximality of $x$ we also know that $2x\not\in X$. Therefore either $4x\in D$ or there is some $F\in D$ for which $F-2x\in S$ or $2x\in S$. We wish to show that $2x\in S$, so we will rule out the other two possibilities.
\begin{itemize}
    \item If $4x\in D$ then since $x\in X$ we know that $3x=4x-x\not\in S$. Since $D$ is an anti-chain and $4x\in D$ we know that $6x\not\in D$. Next if there is some $F\in D$ for which $F-3x\in S$, then $D$ being an anti-chain implies that $F\neq 3x,2x$. Then $F-3x$ is a non-zero element of $S$. By the previous casework this implies $F-2x=x+(F-3x)\in S$. Next $F-2x$ is a non-zero element of $S$ and hence by another application of it $F-x=x+(F-2x)\in S$. But this is impossible since $x\in X$. Therefore $\forall F\in D$: $F-3x\not\in S$. This means that $3x\in X$ and it contradicts the maximality of $x$.
    \item If there is some $F\in D$ for which $F-2x\in S$. Then firstly since $x\in X$ we know that $2x\not\in D$ so $F-2x\neq 0$. This means that $F-x=x+(F-2x)\in S$. But this contradicts the fact that $x\in X$.
\end{itemize}
We therefore conclude that $2x\in S$. This shows that $S_2$ is closed under addition and hence is a GNS. Since $\forall F\in D: F-x\not\in S$ we know that $\forall F\in D: F\neq x$. Therefore $S_2\cap D=\emptyset$. This means that $S_2$ is in the collection and $S\subsetneq S_2$. This contradicts the maximality of $S$ in the collection. Therefore $X$ must be empty and hence $S$ is quasi-irreducible with $FA(S)=D$.

\end{proof}

The special case of this theorem when $|D|=1$ was proved in \cite{Irr}, they call such GNS irreducible and study them. A GNS is irreducible if and only if $\tau(S)=1$ and it is quasi irreducible.
We now prove the second half of Theorem \ref{char Irr}. The special case of this when $\tau(S)=1$ was also proved in \cite{Irr}.

\begin{proposition}\label{quasi-irreducible PF}
A GNS $S$ is quasi-irreducible if and only if it satisfies the property that for every $P\in PF(S)$ either $P\in FA(S)$ or $2P\in FA(S)$.
\end{proposition}
\begin{proof}
First suppose that $S$ is quasi-irreducible. Consider some $P\in PF(S)$. Since $P\in\Hs(S)$, we know that either $2P\in FA(S)$ or there is some $F\in FA(S)$ for which $F-P\in S$. We have nothing to prove in the first case. So suppose that there is some $F\in FA(S)$ for which $F-P\in S$. If $P\neq F$, then $F-P$ is a non-zero element of $S$. Since $P\in PF(S)$, this would imply that $F=P+(F-P)\in S$. This is impossible. Therefore $P=F$, in particular $P\in FA(S)$.

Conversely suppose all pseudo Frobenius gap of $S$ satisfy the given property. Assume for the sake of contradiction that $S$ is not quasi-irreducible. We know that $FA(S)$ is an anti-chain in $\Nd$. Therefore by Theorem \ref{quasi irreducible} there is some $S'\supsetneq S$ for which $FA(S')=FA(S)$. Let $P$ be a maximal point of $S'\setminus S$ under the natural partial ordering. Since $S'$ is a GNS we know that $2P\in S'$. Moreover, by the maximality of $P$, we know that $2P\not\in (S'\setminus S)$. This means that $2P\in S$. Similarly given a non-zero $x\in S$, we know that $P+x\in S'$, but $P+x\not\in (S'\setminus S)$. Therefore $P+x\in S$. This means that $P$ is a pseudo Frobenius gap of $S$. Now $P$ must satisfy the given property, but we have seen that $2P\in S$. Therefore there is some $F\in FA(S)$ for which $F-P\in S$. However in this case $F-P\in S'$ and hence $F=P+(F-P)\in S'$. This contradicts the fact that $FA(S')=FA(S)$. We therefore conclude that $S$ must be quasi-irreducible.
\end{proof}

\begin{corollary}
If $S$ is a quasi-irreducible GNS, then
$$\tau(S)\leq t(S)\leq 2 \tau(S).$$
\end{corollary}

In the case of Frobenius GNS the characterisation is actually a bit stronger. The authors of \cite{Irr} prove that if $S$ is a Frobenius GNS and at least one coordinate of its Frobenius gap is odd, then $S$ is irreducible if and only if $PF(S)=\{F(S)\}$. In this case the GNS is symmetric.
On the other hand if $S$ is a Frobenius GNS and all coordinates of its Frobenius gap are even, then $S$ is irreducible if and only if $PF(S)=\{\frac{F(S)}{2},F(S)\}$. In this case the GNS is pseudo-symmetric.
\begin{example}
One might wonder if this stronger characterisation can be extended to the case when $\tau(S)>1$. One might guess that if $S$ is a quasi-irreducible GNS, then $$PF(S)=FA(S)\cup \{P\in\Nd\mid 2P\in FA(S)\}.$$
However this is not the case. Consider $S\subseteq\mathbb{N}^2$, with $$\Hs(S)=\{(1,0),(2,0),(0,1),(1,1),(2,2),(1,3)\}.$$
It is seen that this is indeed a GNS and $PF(S)=FA(S)=\{(2,2),(1,3)\}$. This means that $S$ is quasi-symmetric and hence quasi-irreducible. However, $2(1,1)\in FA(S)$ and $(1,1)\not\in PF(S)$.
\end{example}

\section{Frobenius GNS of small and large type}

This section deals with the $t(S)$ for Frobenius GNS. It is known that for any Frobenius GNS $S$
$$t(S)\leq 2g(S)+1-\|F(S)\|.$$
And a Frobenius GNS is called almost symmetric if $t(S)= 2g(S)+1-\|F(S)\|$. We establish a new property in Proposition \ref{Condition almost sym} that is equivalent to almost-symmetry. We then find the lower bound on the type of a Frobenius GNS given its genus and Frobenius gap. We show that
$$\frac{g(S)}{\|F(S)\|-g(S)}\leq t(S).$$

\begin{proposition}\label{Condition almost sym}
A Frobenius GNS $S$ with Frobenius gap $F$ is almost symmetric if and only if
$$T(S)=\{x\in\Nd\mid F-x\in \left(\left(\mathbb{Z}^d\setminus S\right)\cup\{0\}\right)\}$$
is a GNS.
\end{proposition}
\begin{proof}
Firstly note that
$$\Nd\setminus T(S)=\{F-s\mid s\in S\setminus\{0\}, s\leq F\}.$$
So $|\Nd\setminus T(S)|=\|F\|-g(S)-1$. Now consider some $x\in\Nd$.
We see that $x+T(S)\subseteq T(S)$ if and only if $y\not\in T(S)$ implies that $y-x\not\in T(S)$. This happens if and only if $F-y\in S\setminus\{0\}$ implies $F-y+x\in S\setminus\{0\}$. This is clearly equivalent to $x\in S\cup PF(S)$. This means that
$$A=\{x\in\Nd\mid x+T(S)\subseteq T(S)\}=S\cup PF(S).$$
Now since $t(S)$ is the size of $PF(S)$ we see that
$$t(S)=|A\setminus S|=|\Nd\setminus S|-|\Nd\setminus A|=g(S)-|\Nd\setminus A|.$$
Since $0\in T(S)$ we know that $A\subseteq T(S)$. This implies that
$$t(S) =g(S)-|\Nd\setminus A| \leq g(S)-|\Nd\setminus T(S)| =2g(S)+1-\|F\|.$$
Moreover equality holds if and only if $T(S)=A$. This is equivalent to $T(S)$ being closed under addition, which is of course equivalent to $T(S)$ being a GNS.
\end{proof}

\begin{theorem}\label{lower bound for type}
Given a Frobenius GNS $S\subseteq\Nd$ we have
$$\frac{g(S)}{\|F(S)\|-g(S)}\leq t(S).$$
\end{theorem}
\begin{proof}
Fix a relaxed monomial ordering $\prec$ on $\Nd$.
Define a map $\phi$ from $\Hs(S)$ to $S$ as follows:
$$\phi(x)=\max_{\prec}\{s\in S\mid x+s\in\Hs(S)\}.$$
Here we are taking the maximum of a finite nonempty set, so $\phi$ is well defined. Consider some nonzero $s\in S$, we know that $\phi(x)+s\in S$ and $\phi(x)\prec \phi(x)+s$. The maximality of $\phi(x)$ implies that $x+\phi(x)+s\in S$. This means that $x+\phi(x)\in PF(S)$. Let $B$ be the box $$B=\{x\in\Nd\mid 0\leq x\leq F(S)\}.$$
So $|B|=\|F(S)\|$. Since $S$ is a Frobenius GNS, we know that $\Hs(S)\subseteq B$. This means that $|B\cap S|=\|F(S)\|-g(S)$. Now we define a map $\psi$ from $\Hs(S)$ to $(S\cap B)\times PF(S)$ given by
$$\psi(x)=(\phi(x),x+\phi(x)).$$
This map is clearly injective therefore $g(S)\leq (\|F(S)\|-g(S))t(S)$.
\end{proof}

\section{Lower bounds for the number of Frobenius GNSs}

In this and the next section we will attempt to count the number of Frobenius GNS with a given Frobenius number in $\Nd$. In this section we will obtain a lower bound for $N(F)$. We denote $\overline{x}=\lfloor\frac{x+1}{2}\rfloor$.

First consider the case when $d=1$ i.e. of numerical semigroups. Given $F\in\mathbb{N}$, let $B=\{x\in\mathbb{N}\mid \frac{F}{2}<x<F\}$. So $|B|=\overline{F-1}$. Now for any subset $X\subseteq B$, let $S(X)=\{0\}\cup X\cup \{x\mid x>F\}$. Then $S(X)$ is closed under addition and hence is a numerical semigroup. Moreover distinct $X$ lead to distinct numerical semigroups. We can therefore conclude that for $F\in \mathbb{N}$, $N(F)\geq 2^{\overline{F-1}}$. We will extend this technique to higher $d$, by choosing a large piece of the box where we can pick points almost independently.

Given $F\in\Nd$ we denote
$$S_F=\{0\}\cup\{(x^{(1)},\dots,x^{(d)})\mid \exists i: x^{(i)}>F^{(i)}\}.$$

\begin{theorem}\label{main lb}
Let $d_1=\lceil\frac{d+1}{3}\rceil$. If $F\in\Nd$ then
$$\left(3^{\frac{1}{2}\sum_{i=d_1}^{d-d_1}\binom{d}{i}}2^{\sum_{i=d-d_1+1}^{2d_1-1}\binom{d}{i}}\right)^{\overline{F^{(1)}}\cdots\overline{F^{(d)}}}\leq N(F).$$
\end{theorem}
\begin{proof}
For a subset $A\subseteq\{1,2,\dots,d\}$, consider the following box
$$B_A=\left\{x\in\Nd\;\Big|\; \text{For }i \in A: \frac{F^{(i)}}{2}< x^{(i)}\leq F^{(i)},\text{For }i \notin A: 0\leq x^{(i)}<\frac{F^{(i)}}{2}\right\}.$$
For each $A$ the size of the box is $$|B_A|=\overline{F^{(1)}}\dotsm\overline{F^{(d)}}.$$
Let $B$ be the union of $B_A$ for all subsets $A$ with size $d_1\leq |A|\leq d-d_1$. Since all boxes have the same size, the size of $B$ is
$$|B|=\overline{F^{(1)}}\dotsm\overline{F^{(d)}}\sum_{i=d_1}^{d-d_1}\binom{d}{i}.$$
Let $C$ be the union of $B_A$ for all subsets $A$ with size $d-d_1+1\leq |A|\leq 2d_1-1$. So $C$ is disjoint from $B$ and the size of $C$ is
$$|C|=\overline{F^{(1)}}\dotsm\overline{F^{(d)}}\sum_{i=d-d_1+1}^{2d_1-1}\binom{d}{i}.$$
Also note that for any $x\in B$, $F-x$ is also in $B$. And for any $x\in C$, $F-x\not\in (B\cup C)$. A subset of $Y\subseteq B$ is called good if $x\in Y$ implies $F-x\not\in Y$. There are $3^{\frac{1}{2}|B|}$ good subsets of $B$. 
For a good subset $Y$ of $B$ and any subset $Z$ of $C$ we let $X=Y\cup Z$ and define
$$S(X)=S(Y,Z)=S_{F}\cup X\cup (X+X).$$
Since $Y$ was a good subset we know that $F$ is not in $S(X)$. It is therefore clear that $F$ is the unique maximal element of $\Nd\setminus S(X)$ under the natural partial ordering. 

We next show that $S(X)$ is closed under addition. Consider non-zero $x,y\in S(X)$. If at least one of them is in $S_F$ then $x+y$ is also in $S_F$. Therefore suppose neither of them is in $S_F$. If both of them are in $X$ then $x+y$ is in $X+X$. The remaining cases are when one of them is in $X$ and the other in $X+X$ or when both are in $X+X$. We can therefore write $x+y=\sum_{i=1}^{n}x_i$ with $x_i\in X$, $n=3$ in the first case and $n=4$ in the second. Say $x_i$ is in $B_{A_i}$ with $|A_i|\geq d_1$ for $1\leq i\leq n$. Since $$|A_1|+|A_2|+|A_3|\geq 3d_1>d,$$ we know that $A_1,A_2,A_3$ cannot be pairwise disjoint.
So say $t\in A_1\cap A_2$. Therefore $x_1^{(t)},x_2^{(t)}$ are both bigger than $\frac{F^{(t)}}{2}$ and hence $(x+y)^{(t)}>F^{(t)}$. This implies that $x+y\in S_F\subseteq S$. This shows that $S(X)$ is closed under addition. Therefore $S(X)$ is a Frobenius GNS with Frobenius gap $F$.

Finally we show that $S(X)$ are distinct for distinct $X$. This will follow from the fact that $$S(X)\cap (B\cup C)=X.$$
Clearly $X\subseteq S(X)\cap (B\cup C)$. However if equality doesn't hold then there will be some $x$ in $(X+X)\cap (B\cup C)$. This means that $x=x_1+x_2$ with $x_1,x_2\in X$. Say $x_i\in B_{A_i}$, we know that $d_1\leq|A_i|\leq 2d_1-1$. Now if $A_1\cap A_2\neq\emptyset$ then there will be some $t\in A_1\cap A_2$. That will imply $x^{(t)}>F^{(t)}$ and contradict $x\in (B\cup C)$. On the other hand if $A_1\cap A_2=\emptyset$ then $|A_1\cup A_2| =|A_1|+|A_2| \geq 2d_1>2d_1-1$. For each $i\in (A_1\cup A_2)$ we have $x^{(i)}>\frac{F^{(i)}}{2}$. This means that $x$ cannot be in any $B_A$ with $|A|\leq 2d_1-1$ and this again contradicts $x\in (B\cup C)$. We therefore see that $(X+X)\cap (B\cup C)=\emptyset$ and hence $X=S(X)\cap (B\cup C)$. Therefore $S(X)$ are distinct Frobenius GNS for distinct $X$. Hence the number of Frobenius GNS we constructed is
$3^{\frac{1}{2}|B|}2^{|C|}$.
\end{proof}
\begin{corollary}\label{cor main lb}
Let $d_1=\lceil\frac{d+1}{3}\rceil$. If $F\in\Nd$ then
$$\left(\left(\sqrt{3}\right)^{\frac{1}{2^{d}}\sum_{i=d_1}^{d-d_1}\binom{d}{i}}\times \Big(2\Big)^{\frac{1}{2^d}\sum_{i=d-d_1+1}^{2d_1-1}\binom{d}{i}}\right)^{\|F-1\|}\leq N(F).$$
\end{corollary}

For most $d$, the lower bound in Corollary \ref{cor main lb} appears to be optimised. However for $d=5$ this gives a lower bound of
$$\left(3^{\frac{5}{16}}\right)^{F^{(1)}F^{(2)}F^{(3)}F^{(4)}F^{(5)}}\leq N(F^{(1)},F^{(2)},F^{(3)},F^{(4)},F^{(5)}).$$
But we can improve the $3^{5/16}$ to $\sqrt{2}$.

\begin{proposition}
For $F\in\mathbb{N}^{5}$
$$\left(\sqrt{2}\right)^{\|F-1\|}\leq N(F^{(1)},F^{(2)},F^{(3)},F^{(4)},F^{(5)}).$$
\end{proposition}
\begin{proof}
We use the notations from the proof of Theorem \ref{main lb}. Let $D$ be the union of all boxes $B_A$ with $|A|\geq 3$. There are $16$ such boxes, so the size of $D$ is
$$|D|=16\overline{F^{(1)}}\dots\overline{F^{(5)}}.$$
For an arbitrary subset $X$ of $D$ we define
$$S(X)=S_{F}\cup X$$
We see that $F$ is not in $S(X)$. It is therefore clear that $F$ is the unique maximal element of $\Nd\setminus S(X)$ under the natural partial ordering. 

We next show that $S(X)$ is closed under addition. Consider non-zero $x,y\in S(X)$. If at least one of them is in $S_F$ then $x+y$ is also in $S_F$. Therefore suppose that both of them are in $X$. Say $x\in B_{A_1}$ and $y\in B_{A_2}$ with $|A_i|\geq 3$. Since $|A_1|+|A_2|\geq 6>5$, we know that $A_1$ and $A_2$ cannot be disjoint.
Say $t\in A_1\cap A_2$, then $(x+y)^{(t)}>F^{(t)}$. This implies that $x+y\in S_F\subseteq S$. This shows that $S(X)$ is closed under addition. Therefore $S(X)$ is a Frobenius GNS with Frobenius gap $F$.

Finally we see that $S(X)$ are distinct for distinct $X$. This follows from the fact that $$S(X)\cap D=X.$$
Hence the number of Frobenius GNS we constructed is
$2^{|D|}$. Finally notice that
$$2^{|D|}=2^{16\overline{F^{(1)}}\dots\overline{F^{(5)}}}\geq 2^{16\frac{F^{(1)}\dots F^{(5)}}{32}}=\sqrt{2}^{\|F-1\|}.$$
\end{proof}

\begin{lemma}
For $d\in\mathbb{N}$ let $d_1=\lceil\frac{d+1}{3}\rceil$. Then we have:
$$\lim_{d\to\infty}\frac{1}{2^{d}}\sum_{i=d_1}^{d-d_1}\binom{d}{i}=1.$$
\end{lemma}
\begin{proof}
From Hoeffding's inequality \cite{Hoeffding} we see that
$$\frac{1}{2^{d}}\sum_{i=0}^{d_1-1}\binom{d}{i} =\frac{1}{2^{d}}\sum_{i\leq \frac{d}{3}}\binom{d}{i}\leq \exp{\left(-2d\left(\frac{1}{6}\right)^2\right)}.$$
Therefore
$$\frac{1}{2^{d}}\sum_{i=d_1}^{d-d_1}\binom{d}{i} =1-2\frac{1}{2^{d}}\sum_{i=0}^{d_1-1}\binom{d}{i} \geq 1-2\exp{\left(-\frac{d}{18}\right)}.$$
We conclude that
$$\lim_{d\to\infty}\frac{1}{2^{d}}\sum_{i=d_1}^{d-d_1}\binom{d}{i}=1.$$
\end{proof}

\begin{corollary}\label{large d lower bound}
Given $\epsilon>0$, for sufficiently large $d$ we have:\\ for every $F\in\Nd$
$$\left(\sqrt{3}-\epsilon\right)^{\|F-1\|}\leq N(F).$$
\end{corollary}


\section{Upper bounds for the number of Frobenius GNSs}

In this section we will obtain an upper bound for the number of Frobenius GNS with a given Frobenius gap.

\begin{lemma}\label{root 3 upper bound}
For any $F\in\Nd$ we have:
$$N(F)\leq \sqrt{3}^{\|F\|}.$$
\end{lemma}
\begin{proof}
Consider the box $X=\{x\in\Nd\mid 0\leq x\leq F\}$. This box has $\|F\|$ points in it. They are divided into $\lfloor\frac{\|F\|}{2}\rfloor$ pairs of the form $x,F-x$ and possibly a single point $x$ with $x+x=F$. If $S$ is a Frobenius GNS with Frobenius gap $F$, then $S_F\subseteq S$ and $\frac{F}{2}\not\in S$. Moreover for each of the $\lfloor\frac{\|F\|}{2}\rfloor$ pairs we can either pick one of the two points or neither of them. This gives $3$ choices for each pair. Therefore
$$N(F)\leq 3^{\lfloor\frac{\|F\|}{2}\rfloor}\leq \sqrt{3}^{\|F\|}.$$
\end{proof}

Combining Corollary \ref{large d lower bound} and Lemma \ref{root 3 upper bound}, we get the following result.

\begin{theorem}\label{count large d}
Given $\epsilon>0$, there exist $M>0$ such that for $d>M$ and $F\in\Nd$
$$\left(\sqrt{3}-\epsilon\right)^{\|F-1\|}\leq N(F)\leq \sqrt{3}^{\|F\|}.$$
\end{theorem}

Theorem \ref{count large d} shows that for large $d$ the upper bound of $\sqrt{3}^{\|F\|}$ is close to the actual value. However for small $d$ this is not the case. For example for $d=1$, \cite{Backelin} proves that for any $F\in\mathbb{N}$
$$N(F)\leq 4\sqrt{2}^{F}.$$
We therefore look for a stronger upper bound specially for smaller $d$.

Given $P,F\in\Nd$ with $P\leq F$ we denote by $L(P,F)$ the number of Frobenius GNS with Frobenius gap $F$ that do not contain $P$.

\begin{lemma}\label{F no P}
Given $P,F\in\Nd$ with $P\leq F$, we have
$$L(P,F)\leq \phi^{\|F\|}\left(\frac{\phi}{\sqrt[4]{5}}\right)^{\|F\|-\|P\|}.$$
\end{lemma}
\begin{proof}
Consider the box $B=\{x\in\Nd\mid x\leq F\}$. It has $\|F\|$ points. Consider a graph with points of $B$ as vertices, such that $x$ and $y$ are connected by an edge when $x+y$ is $P$ or $F$. For any Frobenius GNS $S$ with $F(S)=F$ and $P\in\Hs(S)$, we know that $S\cap B$ will be a subset of the graph that do not contain any edges within it. We therefore count such subsets to get an upper bound for $L(P,F)$.

Each $x\in B$ has one edge connecting it to $F-x$. Those $x$ that satisfy $x\leq P$ also have an edge to $P-x$. Therefore there are $\|P\|$ vertices of degree two and $\|F\|-\|P\|$ vertices of degree one. A graph in which all vertices have degree $1$ or $2$ is a disjoint union of paths and cycles. Let $k=\frac{\|F\|-\|P\|}{2}$. The graph would have $k$ disjoint path graphs and some disjoint cycles. Say there are $l$ cycles. Say the lengths of the paths are $n_1,n_2,\dots,n_k$ and the lengths of the cycles are $m_1,\dots,m_l$. Then
$$\sum_{i=1}^k n_i+\sum_{i=1}^lm_i=\|F\|.$$

We call a subset of the vertices of a graph good if the subset doesn't contain any edges.
We claim that for a path graph with $n$ vertices there are $F_{n+2}$ good subsets. Here $F_k$ is the $k^{th}$ Fibonacci number. This is easily seen for $n=1,2$. We proceed by induction. Suppose that $n\geq 3$ and this has been checked for $n-1$, $n-2$. Now consider a path graph of length $n$, call the vertices $x_1,\dots,x_n$ in order. If a good subset includes $x_n$ then it cannot include $x_{n-1}$, so there are $F_{(n-2)+2}$ good subsets that include $x_n$. On the other hand there are $F_{(n-1)+2}$ good subsets that do not have $x_n$. Therefore the total number of good subsets is $F_{n}+F_{n+1}=F_{n+2}$. This completes the induction step.

Now we consider a cycle of length $n$. Say $x_1,x_2,x_3$ are three consecutive vertices. If a good subset contains $x_2$, then it cannot contain $x_1$ and $x_3$. Therefore there are $F_{(n-3)+2}$ good subsets that contain $x_2$. And there are $F_{(n-1)+2}$ good subsets that do not contain $x_2$. Therefore the total number of good subsets of a cycle of length $n$ is $F_{n-1}+F_{n+1}$.
$$F_{n-1}+F_{n+1}\leq \frac{1}{\sqrt{5}}\left(\phi^{n-1}+\phi^{n+1}\right) =\frac{\phi+\phi^{-1}}{\sqrt{5}}\phi^n=\phi^n.$$

We therefore see that the number of good subsets of our original graph is
$$\prod_{i=1}^{k}F_{n_i+2}\prod_{i=1}^l(F_{m_i-1}+F_{m_i+1}) \leq \prod_{i=1}^{k}\frac{1}{\sqrt{5}}\phi^{n_i+2} \prod_{i=1}^l\phi^{m_i} =\left(\frac{\phi^2}{\sqrt{5}}\right)^k\phi^{\|F\|}.$$

\end{proof}

We now obtain an improved upper bound of $N(F)$ by combining Lemma \ref{root 3 upper bound} and Lemma \ref{F no P}, while keeping in mind that Lemma \ref{F no P} is more accurate when $\|F\|-\|P\|$ is small.

\begin{lemma}\label{upper bound epsilon}
For any $\epsilon$ with $0<\epsilon<1$ and for any $F\in\Nd$ we have
$$N(F)\leq \sqrt{3}^{(1-2\epsilon^d)\|F\|}+\epsilon^d\|F\|\phi^{\|F\|}\left(\frac{\phi}{\sqrt[4]{5}}\right)^{(1-(1-\epsilon)^d)\|F\|}.$$
\end{lemma}
\begin{proof}
Denote by $B$ the box consisting of those $x\in\Nd$ with $x\leq F$. Let $B_1$ be the box of those $x$ with $(1-\epsilon)F\leq x\leq F$. And let $B_2$ be the box of those $x$ with $x\leq \epsilon F$.
We divide the Frobenius GNS with $F(S)=F$ into two categories. The first one consisting of those that have at least one gap in $B_1$ (other than $F$) and the second one consisting of those that have no gaps in $B_1$ (except $F$).

First we count the first category GNS. There are $\epsilon^d\|F\|$ points in $B_1$. For each $P\in B_1$ the number of first category GNS with $P$ as a gap is at most
$$\phi^{\|F\|}\left(\frac{\phi}{\sqrt[4]{5}}\right)^{\|F\|-\|P\|} \leq \phi^{\|F\|}\left(\frac{\phi}{\sqrt[4]{5}}\right)^{\|F\|(1-(1-\epsilon)^d)}.$$
Therefore the total number of first category GNS is at most
$$\epsilon^d\|F\|\phi^{\|F\|}\left(\frac{\phi}{\sqrt[4]{5}}\right)^{\|F\|(1-(1-\epsilon)^d)}.$$

We now count the second category GNS. A second category GNS must contain all of $B_1$ (except $F$) and hence cannot intersect $B_2$ (other than $0$). There are $(1-2\epsilon^d)\|F\|$ points in $B\setminus (B_1\cup B_2)$. They can be divided into pairs of the form $x,F-x$. A GNS cannot have both the points from any of these pairs. Therefore we have 3 choices for each pair and the total number of second category GNS is at most
$$\sqrt{3}^{(1-2\epsilon^d)\|F\|}.$$
\end{proof}

We now optimise the $\epsilon$ in Lemma \ref{upper bound epsilon}.

\begin{proposition}\label{bd}
Let $\epsilon_d$ be the solution of the equation:
$$(1-\epsilon_d)^d\log\left(\frac{\phi}{\sqrt[4]{5}}\right)-\epsilon_d^d=\log\left(\frac{\phi^2}{\sqrt[4]{5}\sqrt{3}}\right).$$
And let $b_d=\sqrt{3}^{(1-2\epsilon_d^d)}$ then
$$N(F)\leq \epsilon_d^d\|F\|b_d^{\|F\|}.$$
\end{proposition}

\section{Further Questions}

While we have obtained a good estimate of $N(F)$ for large $d$, our upper bound is quite weak for smaller $d$.
For $d\neq 5$, let $d_1=\lceil\frac{d+1}{3}\rceil$ and $$a_d=\left(\left(\sqrt{3}\right)^{\frac{1}{2^{d}}\sum_{i=d_1}^{d-d_1}\binom{d}{i}}\times \Big(2\Big)^{\frac{1}{2^d}\sum_{i=d-d_1+1}^{2d_1-1}\binom{d}{i}}\right)$$
For $d=5$ let $a_5=\sqrt{2}$. And $b_d$ are the constants from Proposition \ref{bd}.
We have shown that for each $F\in \Nd$
$$a_d^{\|F-1\|}\leq N(F)\leq O\left(\|F\|b_d^{\|F\|}\right).$$
The constants $\sqrt{2}\leq a_d\leq b_d< \sqrt{3}$ satisfy
$$\lim_{d\to\infty}a_d =\lim_{d\to\infty}b_d =\sqrt{3}.$$
Some of these constants are listed in Table \ref{tab: lb ub} up to four decimal places.

\begin{conjecture}
For each $d\in\mathbb{N}_{>0}$, $N(F)$ is of the magnitude of $a_d^{\|F\|}$.
\end{conjecture}

\begin{table}
    \centering
    \begin{tabular}{|c|c|c|c|c|c|c|c|c|}
    \hline
        d & $a_d$ & $b_d$ &d & $a_d$ & $b_d$ &d & $a_d$ & $b_d$\\
        \hline
        1 &1.4142 &1.4142 &6 &1.4904 &1.7311 &11 &1.5293 &1.7320\\
        2 &1.3160 &1.6630 &7 &1.5130 &1.7319 &12 &1.5798 &1.7320\\
        3 &1.4142 &1.6968 &8 &1.4777 &1.7320 &13 &1.5891 &1.7320\\
        4 &1.4612 &1.7173 &9 &1.5415 &1.7320 &14 &1.5693 &1.7320\\
        5 &1.4142 &1.7275 &10 &1.5553 &1.7320 &15 &1.6095 &1.7320\\
        \hline
    \end{tabular}
    \caption{Constants for upper and lower bound}
    \label{tab: lb ub}
\end{table}

Another direction to extend this would be to consider an anti-chain $A$ of $k$ points in $\Nd$ (anti-chain with respect to natural partial ordering). And attempting to count the number of GNS $S\subseteq \Nd$ for which $FA(S)=A$.

\section{Acknowledgements}
We would like to thank Nathan Kaplan for several helpful discussions on the subject.


\begin{thebibliography}{}
\bibitem{Def F}
G. Failla, C. Peterson, R. Utano. (2016).
Algorithms and basic asymptotics for generalized numerical semigroups in ${\mathbb {N}}^d$.
\emph{Semigroup Forum} 92, 460–473. https://doi.org/10.1007/s00233-015-9690-8

\bibitem{Irr}
C. Cisto, G. Failla, C. Peterson, et al. (2019). Irreducible generalized numerical semigroups and uniqueness of the Frobenius element.
\emph{Semigroup Forum}
99, 481–495
https://doi.org/10.1007/s00233-019-10040-1

\bibitem{Almost Irr}
C. Cisto, W. Tenorio. (2021).
On almost-symmetry in generalized numerical semigroups.
Communications in Algebra, 49, no. 6, 2337–2355

\bibitem{Backelin}
J. Backelin. (1990). On the number of semigroups of natural numbers.
\emph{Mathematica Scandinavica}, 66, 197-215.

\bibitem{Robbiano}
L. Robbiano. (1986).
On the theory of graded structures.
\emph{Journal of Symbolic Computation},
Volume 2, Issue 2,
Pages 139-170,
ISSN 0747-7171,

\bibitem{Hoeffding}
W. Hoeffding. (1963).
Probability Inequalities for Sums of Bounded Random Variables.
\emph{Journal of the American Statistical Association}, 58(301), 13-30.

\bibitem{Gen ref 1}
J.C. Rosales, P.A. Garcia Sanchez. (2009).
\emph{Numerical  semigroups, Developments in Mathematics},
vol 20, Springer.

\bibitem{Gen ref 2}
A. Assi and P.A. Garcia-Sanchez. (2016),
\emph{Numerical  semigroups  and  applications, RSME Springer Series},
vol 1, Springer, 2016.

\bibitem{Type}
R. Froberg, C. Gottlieb, R. Häggkvist (1986).
\emph{On numerical semigroups}.
Semigroup Forum. 35. 63-83. 10.1007/BF02573091. 
\bibitem{H Nari}
H. Nari (2011).
\emph{Symmetries on almost symmetric numerical semigroups}.
Semigroup Forum. 86. 10.1007/s00233-012-9397-z. 

\bibitem{My paper type}
D. Singhal (2021).
Numerical Semigroups of small and large type.
International Journal of Algebra and Computation.
10.1142/S0218196721500417

\bibitem{gen GNS}
C. Cisto, G. Failla, R. Utano. (2019).
On the generators of a generalized numerical semigroup.
Analele Universitatii "Ovidius" Constanta - Seria Matematica, 27(1) 49-59. https://doi.org/10.2478/auom-2019-0003

\bibitem{}
J.I. García-García, I. Ojeda, J.C. Rosales et al. (2020).
On pseudo-Frobenius elements of submonoids of $\Nd$.
Collect. Math. 71, 189–204
https://doi.org/10.1007/s13348-019-00267-0

\bibitem{}
C. Cisto, M. Delgado, P. A. Garciıa-Sanchez,
Algorithms for generalized numerical semigroups,
J. Algebra Appl. 0(0), no. 0, 2150079
https://doi.org/10.1142/S0219498821500791



\end{thebibliography}
\end{document}